\documentclass[11pt] {amsart}
\usepackage{amscd, amssymb}
\usepackage{graphics}

\theoremstyle{plain}

\numberwithin{equation}{section}

\newtheorem{theorem}[equation]{Theorem}

\newtheorem{proposition}[equation]{Proposition}

\newtheorem{lemma}[equation]{Lemma}
\newtheorem{corollary}[equation]{Corollary}
\theoremstyle{definition}
\newtheorem{remark}[equation]{Remark}
\newtheorem{example}[equation]{Example}
\newtheorem{definition}[equation]{Definition}

\def    \Hat  {\widehat}

\def    \Gh     {\widehat{G}}
\def    \Th     {\widehat{T}}

\newcommand{\printname}[1]

\def    \R  {{\Bbb R}}
\def    \Z  {{\Bbb Z}}

\def    \RP {{\Bbb {RP}}}
\def    \C  {{\Bbb C}}

\def    \Tilde  {\widetilde}
\begin{document}
\title[The fundamental group of $G$-manifolds]{The fundamental group of $G$-manifolds}
\author{Hui Li}
\address{School of mathematical Sciences, Box 173,
        Suzhou University, Suzhou, 215006, China.}
\email{hui.li@suda.edu.cn}

\begin{abstract}
 Let $G$ be a connected compact Lie group, and
let $M$ be a connected Hamiltonian $G$-manifold with
equivariant moment map $\phi$.

 We prove that if there is a simply connected orbit $G\cdot x$, then
$\pi_1\left(M\right)\cong\pi_1\left(M/G\right)$; if additionally
$\phi$ is proper, then
$\pi_1\left(M\right)\cong\pi_1\left(\phi^{-1}(G\cdot a)\right)$,
where $a=\phi(x)$.

   We also prove that if a maximal torus of $G$ has a fixed point $x$,
 then $\pi_1\left(M\right)\cong\pi_1\left(M/K\right)$,  where $K$ is any
 connected subgroup of $G$;  if additionally
 $\phi$ is proper, then
 $\pi_1\left(M\right)\cong\pi_1\left(\phi^{-1}(G\cdot a)\right)\cong\pi_1\left(\phi^{-1}(a)\right)$,
 where $a=\phi(x)$.

 Furthermore, we prove that  if $\phi$ is proper, then $\pi_1\big(M/\Gh\big)\cong\pi_1\big(\phi^{-1}(G\cdot
a)/\Gh\big)$ for all $a\in\phi(M)$, where $\Gh$ is any connected
subgroup of $G$ which contains the identity component of each
stabilizer group. In particular,
$\pi_1\left(M/G\right)\cong\pi_1\left(\phi^{-1}(G\cdot a)/G \right)$
for all $a\in\phi(M)$.
\end{abstract}
 \maketitle
 \section{Introduction}

  Let $M$ be a smooth manifold. Let a connected compact Lie group $G$ act
  on $M$. We call $M$ a {\bf $G$-manifold}.

  Let $(M, \omega)$ be a symplectic manifold. Assume that a connected compact Lie group $G$
  acts on $M$  with  moment map  $\phi\colon M\rightarrow \mathfrak{g}^*$,
   where $\mathfrak{g}^*$ is the dual of the Lie algebra of $G$. In this case, we call $(M, \omega)$ a {\bf Hamiltonian $G$-manifold}.
 We will always assume that $\phi$ is equivariant with respect to the $G$ action,
 where $G$ acts on $\mathfrak{g}^*$ by the coadjoint action. Given a value
 $a$ in $\mathfrak{g}^*$, the space ${\bf M_{a}}=\phi^{-1}(G\cdot a)/G$ is called the {\bf symplectic quotient}
  or the {\bf reduced space} at the coadjoint orbit $G\cdot a$. If $G_a$ is the stabilizer group of $a$ under
  the coadjoint action, then we also have that $M_a = \phi^{-1}(a)/G_a$.

 In this paper, unless otherwise stated, $G$ always denotes a connected compact Lie group,
 $\mbox{im}(\phi)$ means the image of $\phi$, and
 {\bf fundamental group} always means the fundamental
 group of a space as a topological space.

 For a compact Hamiltonian $G$-manifold $M$,  we proved the following
 results, which combine Theorem 0.1 in \cite{L1} and Theorems 1.2, 1.3 and 1.6 in \cite{L2}.

  \begin{theorem}\label{thm1}
     Let $(M, \omega)$ be a connected compact Hamiltonian $G$-manifold
  with  moment map $\phi$, where $G$ is a connected compact
   Lie group. Then $\pi_1\left(M\right)\cong \pi_1\left(M/G\right)\cong\pi_1\left(M_a\right)$ for all
  $a\in \mbox{im}(\phi)$.
   \end{theorem}

  Here is a counter example to Theorem~\ref{thm1} when $M$  is not compact.
\begin{example}\label{ex}
 Let $M=S^1\times\R$, and let $S^1$ rotate the first factor; then  the
  moment map is the projection to the second factor $\R$. We have that
  $\pi_1\left(M/S^1\right)\cong\pi_1\left(M_a\right)\ncong\pi_1\left(M\right)$
 for all $a\in \mbox{im}(\phi)$.
\end{example}

 In this paper, we study the fundamental group of noncompact Hamiltonian
 $G$-manifolds. We assume that the moment map is proper. A map is {\bf proper}
  if the inverse image of each compact set is compact.
  Clearly, the moment map of a compact Hamiltonian $G$-manifold is
  proper. The study of the noncompact
  case, which requires new approaches, makes the reasons for the isomorphisms of the fundamental
  groups more clear.

  Let $M$ be a Hamiltonian $G$-manifold. If $M$ is compact, the
  components of the moment map have minima and maxima.
  These allowed us to associate the fundamental group of $M$
  to that of a symplectic quotient at an extremal value (see the outline of proof for more
  detail.). Without assuming that $M$ is compact, we study the induced map by the quotient
 $\pi_1\left(M\right)\to\pi_1\left(M/G\right)$ and obtain the
 following two theorems. In particular, we are able to show
 that this map is an isomorphism if there exists a simply
 connected orbit (compare with Example~\ref{ex}). While for
 compact $M$, the extrema of the components of the moment map provide a
 lot of simply connected orbits.

\begin{theorem}\label{M=M/G}
  Let $(M, \omega)$ be a connected Hamiltonian $G$-manifold with an equivariant moment map $\phi$, where $G$
   is a connected compact Lie group.
 \begin{itemize}
 \item [(A)] If there is a simply connected orbit, then
     $\pi_1\left(M\right)\cong\pi_1\left(M/G\right).$
 \item [(B)] If $\phi$ is proper and if there is a simply
     connected orbit $G\cdot x$, then
 $\pi_1\left(M\right)\cong\pi_1\left(\phi^{-1}(G\cdot
 a)\right)$,  where $a = \phi(x)$.
\end{itemize}
 \end{theorem}

\begin{theorem}\label{M=M/G+}
   Let $(M, \omega)$ be a connected Hamiltonian $G$-manifold with an equivariant moment map $\phi$, where $G$
   is a connected compact Lie group.
  \begin{itemize}
  \item [(A)] If a maximal torus of $G$ has a fixed point,
      then $\pi_1\left(M\right)\cong\pi_1\left(M/K\right)$,
      where $K$ is any connected subgroup of $G$.
  \item [(B)] If $\phi$ is proper and if a maximal torus of
      $G$ has a fixed point $x$, then
      $\pi_1\left(M\right)\cong\pi_1\left(\phi^{-1}(G\cdot
    a)\right)\cong\pi_1\left(\phi^{-1}(a)\right)$,  where
    $a = \phi(x)$.
  \end{itemize}
  \end{theorem}

 Next, we consider the fundamental groups of the different
 quotients.  If $M$ is compact, a maximal
 torus of $G$ has to have a fixed point; so $G$ is the only subgroup  which contains all the stabilizer groups.
 If $M$ is not compact, as Example~\ref{ex}
  shows, the above fact may not be true.  Our next result extends the
  statement that $\pi_1\left(M/G\right)\cong\pi_1\left(M_a\right)$ for all $a\in
\mbox{im}(\phi)$ for compact manifolds to non-compact ones with
proper moment maps.

 \begin{theorem}\label{hatG}
      Let $(M, \omega)$ be a connected Hamiltonian $G$-manifold
      with a proper equivariant moment map $\phi$, where $G$
   is a connected compact Lie group.
Let $\Gh$ be any connected subgroup of $G$ which contains the
identity component of each stabilizer group. Then
$\pi_1\big(M/\Gh\big)\cong\pi_1\big(\phi^{-1}(G\cdot
a)/\Gh\big)$ for all $a\in\mbox{im}(\phi).$ In particular,
$\pi_1\big(M/G\big)\cong\pi_1\big(M_a\big)$ for all
$a\in\mbox{im}(\phi).$
\end{theorem}

When $\mbox{im} (\phi)$ consists of regular values,  all the stabilizer groups are finite.
We have the following corollary of Theorem~\ref{hatG}.

 \begin{corollary}\label{regular}
      Let $(M, \omega)$ be a connected Hamiltonian $G$-manifold
      with a proper equivariant moment map $\phi$, where $G$
   is a connected compact Lie group. If $\mbox{im} (\phi)$ consists of regular
   values, then $\pi_1(M)\cong\pi_1\left(\phi^{-1}(G\cdot a)\right)$ for all $a\in \mbox{im}(\phi).$
   \end{corollary}

\begin{remark}
In Theorems~\ref{M=M/G}, \ref{M=M/G+}, and \ref{hatG}, where the properness of $\phi$ is
assumed, if we replace the properness by the assumption that
$\phi$ is proper as a map onto its image, the same results
still hold. The reason for this is in our proofs of the
theorems, we only need that $\phi$ has a local properness property.
\end{remark}

\begin{remark}
All our isomorphisms of fundamental groups are natural ones.
They are induced either by quotient of a group, or by natural
deformations in the space, or by inclusion of a subset in a
space, where in the last case the isomorphism may not be direct
--- it may be obtained by transitivity of isomorphisms. To
avoid wordy statements, we do not explicitly state the maps in
the theorems.
\end{remark}

In the case  when  $M$  is a  Hamiltonian $S^1$-manifold with a
proper moment map, using Morse theory, Godinho and Sousa-Dias
proved that if the circle action has a fixed point, then  $\pi_1(M)\cong\pi_1(M_a)$ for all
$a\in\mbox{im}(\phi)$  \cite{GSD}.\\

Now, we give a brief description of the ideas of proofs and
compare them with those for compact manifolds.

Let us first recall the proof of Theorem~\ref{thm1} for compact
$M$. In this case, $\mbox{im}(\phi)$ is compact, its
intersection with a closed positive Weyl chamber  is a convex
polytope (Theorem~\ref{convexity}). Let $b$ be a boundary
vertex furthest from the origin on the polytope, then
$\phi^{-1}(b)$ is a fixed set of $G_b$, where $G_b$ is the
stabilizer group of $b$ under the coadjoint action (see Lemma
6.14 in \cite{L2}. Note that $G_b$ contains a maximal torus of
$G$). Hence
$$\pi_1\left(\phi^{-1}(b)\right)\cong\pi_1(M_b).$$
Moreover, $\phi^{-1}(b)$ is the minimum  of the moment map of a
circle subgroup action. By arguments using Morse-Bott theory
(\cite{L1}) for Hamiltonian circle actions, we have
$$\pi_1(M)\cong\pi_1\left(\phi^{-1}(b)\right).$$
These together give the isomorphism
\begin{equation}\label{bdrbb}
\pi_1(M)\cong\pi_1(M_b), \,\,\,\mbox{where}\,\, b \,\,\mbox{is a
boundary value}.
\end{equation}
Next, using local equivariant diffeomorphism, local equivariant
deformation retraction, and a removing process, we proved that,
if $a$ and $a'$ are any two nearby values of $\phi$, then
\begin{equation}\label{aa'}
\pi_1(M_a)\cong\pi_1(M_{a'}).
\end{equation}
Finally, applying the same idea of removing and deformation
retraction alternately on the global space $M/G$, we proved that
there exists a value $c$ such that
\begin{equation}\label{c}
\pi_1(M/G)\cong\pi_1(M_c).
\end{equation}
Theorem~\ref{thm1} follows from (\ref{bdrbb}), (\ref{aa'}) and
(\ref{c}). The hard part of the proof for compact manifolds is
the proof of (\ref{aa'}) and (\ref{c}).

 If we do not assume that $M$ is compact, then the image of $\phi$ may
 not be compact. The proof of (\ref{bdrbb}) does not apply.
 In this paper, we take a new point of view. First, let us mention that
 the method of deforming and removing we used
 to prove (\ref{aa'}) and (\ref{c}) applies to the proof of Theorem~\ref{hatG}. For a more general
 subgroup $\Hat G\subset G$ as in
Theorem~\ref{hatG},  we make a key observation on the links of
the quotient strata to be removed from quotient stratified
spaces (Lemmas~\ref{Tlinksame} and \ref{Glinksame}).
 These observations make the removing process to work.

 For a noncompact Hamiltonian $G$-manifold $M$, without the extrema of the components of the moment map,
 looking at the relation between $\pi_1(M)$ and $\pi_1$ of the different
 quotients is the most different part of the study.
 Our idea is to look at the following induced map of the quotient map
 \begin{equation}\label{quot}
\pi_1\left(M\right)\to\pi_1\left(M/G\right).
 \end{equation}
Since $G$ is connected, (\ref{quot}) is surjective. Note that
if $\pi_1(M) = 1$, then (\ref{quot}) is an isomorphism. More
generally, (\ref{quot}) may not be injective. We will show
 that for any $G$-manifold $M$, if there exist ``enough" simply
connected orbits, then (\ref{quot}) is injective
(Proposition~\ref{propM=M/G}). If $M$ is a compact Hamiltonian
$G$-manifold, due to the existence of the extrema of the
components of the moment map, there are enough simply connected
orbits. In this paper, we succeed in proving that for a
Hamiltonian $G$-manifold $M$,  (\ref{quot}) is injective when
there exists one simply connected orbit. This is the proof of
Theorem~\ref{M=M/G} (A). This last point is a special property
for Hamiltonian $G$-manifolds (see Example~\ref{ex2} for a
counter example).

 Now, we mention a point for the proof of Theorem~\ref{M=M/G+} (A). When $M$ is compact, stabilizer
 groups which contain a maximal torus of $G$ exist. This
 allowed us to prove (\ref{bdrbb}). To use this criterion for noncompact manifolds,
 we will show that in a $G$-manifold, if there
 exists a point $x$ whose stabilizer group contains a maximal torus of $G$, then
 the orbit $G\cdot x$ is simply connected. In particular, this
 is new when the stabilizer group is not connected.

  To prove Theorems~\ref{M=M/G} (B)
  and \ref{M=M/G+} (B), we use part (A) of the corresponding theorem and Theorem~\ref{hatG}.

  \medskip

  Now we give a brief outline of the paper.
  In Section~\ref{section:prelim}, we set up the tools needed for later
  sections. In Section~\ref{section:G-manifolds}, we study
  $G$-manifolds  which have at least one simply connected
  orbit. In Section~\ref{section:hamG-manifolds},  we study Hamiltonian
  $G$-manifolds which have one simply connected orbit.
  Sections~\ref{section:G-manifolds} and \ref{section:hamG-manifolds} are the heart part
  towards proving Theorems~\ref{M=M/G} (A) and \ref{M=M/G+} (A).  In
  Section~\ref{section:M=M/G}, we prove Theorem~\ref{M=M/G}; and in
   Section~\ref{section:M=M/G+}, we prove Theorem~\ref{M=M/G+}.
  In Section~\ref{section:hat}, we prove
  Theorem~\ref{hatG}.

\subsubsection*{Acknowledgment} This work was partially supported by the Postdoctoral
fellowship of the Regional de Bourgogne. This paper was
completed and revised during my research
visits of the Ihes and the Chern Institute. I thank the
institutes for the visiting opportunity and support.

\section{Hamiltonian $G$-manifolds, $G$-manifolds, and the fundamental group of a quotient}
\label{section:prelim}
 This section consists of three subsections. First, we
restrict attention to Hamiltonian $G$-manifolds --- we review a
few theorems on Hamiltonian $G$-manifolds which  will be useful
for our proofs. Next, we consider $G$-manifolds --- we review
stratified spaces and state a lemma on removing strata from
stratified spaces. Our proof of Theorem~\ref{hatG} relies on
this lemma since the quotients are in general stratified
spaces. Finally, we consider an action of a group $K$, which
may not be connected, on a more general space $X$ (which may
not be a smooth manifold), and we state a theorem by Armstrong
on $\pi_1(X/K)$.

\subsection{Hamiltonian $G$-manifolds}
\
\medskip

We first recall the connectivity and convexity theorem.

A subset $\triangle$ of a vector space $V$ is {\bf polyhedral} if it
is the intersection of finitely many closed half spaces, and is {\bf
locally polyhedral} if for each point $x\in\triangle$, there is a
neighborhood $O$ of $x$ in $V$ and a polyhedral set $P$ in $V$ such
that $O\cap\triangle = O\cap P$.
\begin{theorem}\label{convexity}
      (\cite{AT}, \cite{GS0}, \cite{K} and \cite{LMTW})  Let $(M, \omega)$ be a connected
      Hamiltonian $G$-manifold with proper moment map $\phi$, where $G$ is a connected compact Lie group. Let $\mathfrak{t}_+^*$
   be a fixed closed positive Weyl chamber of $\mathfrak{g}^*$. Then
 \begin{enumerate}
  \item for each coadjoint orbit $\mathcal{O}$ in the image of $\phi$, $\phi^{-1}(\mathcal{O})$ is
     connected;
 \item the set
   $\phi(M)\cap\mathfrak{t}_+^*$ is a convex locally polyhedral set,  and it is a convex polytope if $M$ is compact.
  \end{enumerate}
     \end{theorem}

  If $G=T$ is abelian,   $\phi(M)\cap\mathfrak{t}_+^*=\phi(M)$.  It  is a convex locally polyhedral
  set; it consists of {\bf faces}
  with different dimensions. These faces  are caused by different dimensional stabilizer groups
of the action,  and  they  can be in the interior or on the boundary of $\phi(M)$.
  A connected open face is called a {\bf chamber}. By quotienting out a
   subtorus which acts trivially, we can assume that the
  action is effective, and hence there is an open face.  If $G$ is nonabelian, $\phi(M)$ may
 intersect with  different faces  of $\mathfrak{t}_+^*$.  One should distinguish the two
notions of faces.\\

  For a Hamiltonian $G$-manifold with a proper moment map, we can do
  local deformation retractions as follows.

 \begin{theorem}\label{retract} (\cite{W}, \cite{Ler})
           Let $(M, \omega)$ be a connected  Hamiltonian $G$-manifold with proper moment map $\phi$,
           where $G$ is a connected compact Lie
           group.  If $0\in\mbox{im}(\phi)$, then there exists a $G$-invariant open neighborhood
           of $\phi^{-1}(0)$ which $G$-equivariantly deformation retracts to $\phi^{-1}(0)$.
 \end{theorem}

 The local normal form theorem below gives us a description up to
 isomorphism a neighborhood of an isotropic orbit in $M$. An orbit
 is {\bf isotropic} if the restriction of the symplectic form to it
 vanishes.  When an abelian group $T$ acts on $M$, each orbit is isotropic.   We will use this local model to compute the links of
the strata in  quotient stratified spaces.

\begin{theorem}\label{form}
         (\cite{M}, \cite{GS1}) Let $(M, \omega)$ be a symplectic manifold with a Hamiltonian connected compact Lie group $G$ action.
        Let $p\in M$ be a point such that the orbit $G\cdot p$ is isotropic. Let
         $H$ be the stabilizer group of $p$. Then a neighborhood of the orbit $G\cdot p$ in $M$
         is equivariantly symplectomorphic to
         $G\times_H(\mathfrak{h}^{\circ}\times V)$, where $\mathfrak{h}^{\circ}$ is the annihilator of  $\mathfrak{h}=$Lie$(H)$
    in $\mathfrak g^*$ on which $H$ acts by the coadjoint action, and $V$ is
 a complex vector space on which $H$ acts linearly and symplectically.

The $G$ action on this local model is $g_1\cdot [g, a, v]=[g_1g, a,
v]$, and the moment map is
       $\phi([g, a, v])=Ad^*(g)(\phi(p)+a+\psi(v))$, where $\psi(v)$ is the moment map for the $H$ action on $V$.
       \end{theorem}

 Next, we introduce the cross section theorem. In the case when $G$ is
 nonabelian, given a  $G$ orbit $G\cdot x$, it may not be isotropic. But if
 $a=\phi(x)$, and if $G_a$ is the stabilizer group of $a$, then the orbit $G_a\cdot x$ is isotropic in a
 ``cross section'', a symplectic submanifold  of $M$.

 \begin{definition}\label{def 4.1}
   Suppose that a group $G$ acts on a manifold $M$. Given a point $m$ in $M$ with stabilizer group $G_m$,
   a submanifold $U_m\subset M$ containing $m$ is a
   {\bf slice at m} if $U_m$ is $G_m$-invariant, $G\cdot U_m$ is a neighborhood of $m$, and the map
 $$G\times_{G_m}U_m\rightarrow G\cdot U_m, \qquad      [g, u]\longmapsto g\cdot u\quad\mbox{is an isomorphism}.$$
   \end{definition}

   For instance, consider the coadjoint action of $G=SU(2)$ or $SO(3)$ on $\mathbb{R}^3=$Lie$^*(G)$.
   For $a\in \mathbb{R}^3, a\neq 0$, there is a unique ray $I_a$
   connecting $0$ and $a$; the open ray $I_a^{\circ}=I_a-0$ is a slice at $a$.
   If $a=0$, then a  slice at $0$ is  $\mathbb{R}^3$.

    More generally, consider the coadjoint action of a connected compact
  Lie group $G$ on $\mathfrak g^*$.
  Fix a closed positive Weyl chamber $\mathfrak t_+^*$, without loss of generality, take
  $a\in\mathfrak t_+^*$. Let $\tau\subset\mathfrak t_+^*$ be the open face of  $\mathfrak t_+^*$ containing
 $a$  and let $G_a$ be the stabilizer group of $a$. Since all the points on $\tau$ have the same stabilizer group,
 we also use $G_{\tau}$ to denote $G_a$. Then the natural  slice at $a$ is
$U_a=G_a\cdot\{b\in\mathfrak t_+^* \,|\,G_b\subset
G_a\}=G_a\cdot\bigcup_{\tau\subset\overline{\tau'}}\tau'$,
   and it is an open subset of $\mathfrak g_{\tau}^*=\mathfrak g_a^*$.\\

  The following cross section theorem is due to Guillemin and Sternberg.
(Theorem 26.7 in \cite{GS}; for the following version, see Corollary
  2.3.6 in \cite{GLS}.)

    \begin{theorem}\label{cross}
     Let $(M, \omega)$ be a symplectic manifold with a moment map $\phi\colon M\rightarrow \mathfrak{g}^*$ arising from an action of a compact
    connected Lie group $G$. Let $a$ be a point in $\mathfrak{g}^*$ and let $U_a$ be the natural slice at $a$.
     Then the {\bf cross section} $R=\phi^{-1}(U_a)$ is
     a $G_a$-invariant symplectic submanifold of $M$, where $G_a$ is the stabilizer group of $a$. Furthermore,
     the restriction $\phi|_R$ is a moment map
     for the action of $G_a$ on $R$.
     \end{theorem}

  The highest dimensional face $\tau^P$ of $\mathfrak{t}_+^*$ which intersects the image of $\phi$ is called the
   {\bf principal face}. Let $U_{\tau^P}$ be the slice at $\tau^P$. The cross section
 $\phi^{-1}(U_{\tau^P})$ is called the {\bf principal cross section}.
The maximal torus of $G$ acts on the principal cross section (but
not necessarily effectively, see Theorem 3.1 in \cite{LMTW}).

  We will use the cross section theorem in the following  two ways:
\begin{itemize}
 \item   If $a\in\mbox{im}(\phi)\cap\tau$ for some face $\tau\subset\mathfrak t^*_+$, then $\tau$ lies
 on the central dual Lie algebra of $G_a$, so the $G_a$-orbits in $R=\phi^{-1}(U_a)$ are isotropic orbits;
 then we can use the local normal form theorem  in the cross section $R$ to describe a neighborhood of
 a $G_a$-orbit.
 \item
 Assume that the moment map $\phi$ is proper.
 If $a\in\mbox{im}(\phi)\cap\tau$, then by Theorem~\ref{retract},
 there exists a $G_a$-invariant neighborhood of  $\phi^{-1}(a)$  in $R$
 which $G_a$-equivariantly deformation retracts to $\phi^{-1}(a)$. By
 equivariance of $\phi$, there exists a $G$-invariant neighborhood
$\mathcal N$ of $\phi^{-1}(G\cdot a)$ in $M$ which $G$-equivariantly
deformation retracts to $\phi^{-1}(G\cdot a)$.
\end{itemize}

\subsection{G-manifolds and stratified spaces}
\
\medskip

Let $G$, a connected compact Lie group, act on a smooth manifold
$M$, the quotient space $M/G$ is a stratified space (\cite{Br}). If
$(M, \omega)$ is a symplectic manifold and the $G$ action is
Hamiltonian, then the symplectic quotients are stratified spaces
(\cite{SL}).

   Let $X$ be a Hausdorff and paracompact topological space and let $\mathcal J$ be a partially ordered set
   with order relation denoted by $\leq$. A $\mathcal J$-decomposition of $X$ is a locally finite collection
   of disjoint, locally closed manifolds $S_i\subset X$ (one for each $i\in\mathcal J$) called pieces
   such that
 $X=\cup_{i\in\mathcal J}S_i$, and that
 $S_i\cap\overline{S_j}\neq\emptyset\Leftrightarrow S_i\subset\overline{S_j}\Leftrightarrow i\leq j$.
  We call the space $X$ a $\mathcal J$-{\bf decomposed space}.

A stratified space is defined recursively as follows.
 A  decomposed space $X$ is called a {\bf stratified space} if the pieces of $X$, called strata, satisfy the following condition:
Given a point $x$ in a piece $S$ (connected),  there exist an open
neighborhood $\Tilde{U}$ of $x$ in $X$, an open ball $B$ around $x$
in $S$, a compact stratified space $L$, called the {\bf link of x},
and a homeomorphism
   $$\varphi: B\times \overset{\circ}{C}L\rightarrow \Tilde{U}$$
  that preserves the decompositions. Here, $\overset{\circ}{C}L$ is the space obtained by collapsing the boundary $L\times {0}$ of the
  half open cylinder $L\times [0, \infty)$ to a point.
 We also call the  link of $x$ the {\bf link of S}.

    For a smooth manifold $M$, if $N$ is a closed submanifold of $M$, by the tubular neighborhood theorem, a neighborhood of $N$ in $M$
is diffeomorphic to a neighborhood of the normal bundle $E$ of $N$
in $M$. If the fiber of the sphere bundle $S(E)$ of $N$ is connected
and simply connected, i.e., if the codimension of $N$ is bigger than
or equal to $3$, by the Van-Kampen theorem, $\pi_1(M)\cong\pi_1(M-N)$. For a stratified
space $X$, the  link of a stratum $S$  plays the role of the fiber
of the sphere bundle of a submanifold (in a smooth manifold). Using
the Van-Kampen theorem, we can prove the following lemma. May see
Lemma 2.3 in \cite{L2} for a proof.
\begin{lemma}\label{remove:str}
Let $X$ be a stratified space. Let $S$ be a connected closed stratum of $X$ whose link is connected and simply connected.
Then $\pi_1(X)\cong\pi_1\left(X- S\right)$.
\end{lemma}

\subsection{A theorem on the fundamental group of a quotient}
\
\medskip

When a group $K$ which may not be connected act on a space, we
have the following theorem by Armstrong on the fundamental
group of the quotient. We will use this theorem in
Sections~\ref{section:G-manifolds} and \ref{section:hat}.

 \begin{theorem}\label{thmarm}(\cite{A})
   Let $K$ be a compact Lie group acting on a compact path connected and simply connected
   metric space $X$. Let $H$ be the smallest normal subgroup of $K$ which contains the identity component
   of $K$ and all those elements of $K$ which have fixed points.
   Then $\pi_1(X/K)\cong K/H$.
   \end{theorem}

\section{$G$-manifolds which have a simply connected orbit}\label{section:G-manifolds}

 In this section,  we study $G$-manifolds. We first find a criterion on judging injectivity of the
 map $\pi_1(M)\to\pi_1\left(M/G\right)$ induced by the quotient and prove Lemma~\ref{liftpoint}.
 Then we use this to prove that if there are ``enough" simply connected orbits, then the map
 $\pi_1(M)\to\pi_1\left(M/G\right)$ is an isomorphism. Finally,
 we prove that if there exists a stabilizer group which contains
 a maximal torus of $G$, then there exists a simply connected
 orbit.\\

 Let $M$ be a connected smooth $G$-manifold.
 Let $H$ be a subgroup of $G$, and let $(H)$ be the
subgroups of $G$ conjugate to $H$. Denote $M_{(H)}=\{ x\in M\,
|\, G_x\in (H)\}$. The set $M_{(H)}$, if nonempty, is a
submanifold of $M$, and  we call it the {\bf $(H)$-orbit type}
or the {\bf $(H)$-isotropy type}. The manifold $M_{(H)}$ may
not be connected. We denote $M_{(H), c}$ as a connected
component of $M_{(H)}$. There is a {\em connected open dense}
orbit type in $M$, called the {\bf principal orbit type},
denoted $M_P$. If $M_{(H)}\subset \overline{M_{(H')}}$, then
$(H)\supseteq (H')$. Clearly, for each orbit type $M_{(H)}$,
$M_{(H)}\subset\overline{M_P}$.  In each $M_{(H), c}$, by the
slice theorem, $M_{(H), c}$ is a locally trivial fibration over
$M_{(H), c}/G$ with fiber $G/H$; hence the orbits in
the same $M_{(H), c}$ are homotopic. If $M_{(H), c}\subset
\overline{M_{(H'), c}}$, then there is a natural deformation
map from an orbit in $M_{(H'), c}$ to an orbit in $M_{(H), c}$
by going to the closure.

\begin{lemma}\label{liftpoint}
 Let $M$ be a connected  smooth $G$-manifold.
 Then the map $\pi_1(M)\to\pi_1\left(M/G\right)$ induced by the quotient is injective if and only if
 for each point $\bar x\in M/G$, $[\Tilde\alpha_x] = 1\in\pi_1(M)$, where
 $\Tilde\alpha_x$ is any loop which is a lift of the trivial loop $\bar x$.
\end{lemma}

 \begin{proof}
 The only if part is clear.

 Now, we assume that for each point $\bar x\in M/G$, $[\Tilde\alpha_x] = 1\in\pi_1(M)$, where
 $\Tilde\alpha_x$ is any loop which is a lift of the trivial loop $\bar x$.
 Let $\alpha$ be any loop in $M/G$ such that
 $[\alpha]=1\in\pi_1\left(M/G\right)$. Let $\Tilde\alpha\subset M$ be a
 lifted loop of $\alpha$. We want to show that $[\Tilde\alpha] = 1\in
 \pi_1(M)$.  Since $[\alpha]=1\in\pi_1\left(M/G\right)$, the loop $\alpha$
 bounds a disk. We triangulate the disk into small disks which are
 bounded by small loops $\alpha_i$'s, where $i=1, 2, \dots, k$, such that the lift $\Tilde\alpha_i$ of
 each $\alpha_i$
 lies in a small neighborhood $\mathcal N_i$ of an  orbit $\mathcal O_i$.
 Assume that each $\mathcal N_i$ is small enough so that it deformation retracts
 to $\mathcal O_i$; so $\Tilde\alpha_i$ is homotopic to a loop $\Tilde\alpha_i'$ in $\mathcal
 O_i$. Since by our assumption $[\Tilde\alpha_i'] = 1\in \pi_1(M)$, $[\Tilde\alpha_i] = 1\in
 \pi_1(M)$, i.e., $\Tilde\alpha_i$ bounds a disk in $M$. Since this
 is true for all the $\Tilde\alpha_i$'s, $\Tilde\alpha$ bounds a
 disk in $M$.
\end{proof}

\begin{proposition}\label{propM=M/G}
 Let $M$ be a connected  smooth $G$-manifold.
 If each connected component of each orbit type contains a simply connected
 orbit in its closure, then the quotient map induces an isomorphism on $\pi_1$, i.e.,  $\pi_1(M)\cong\pi_1\left(M/G\right).$
 \end{proposition}

\begin{proof}
Since $G$ is connected, we have a surjection
$\pi_1(M)\twoheadrightarrow\pi_1\left(M/G\right).$

Now we prove injectivity.  Let $\bar x\in M/G$ be a trivial
loop. Let $\Tilde \alpha_x$  be a loop in $M$ which is a lift
of $\bar x$. Then $\Tilde \alpha_x$ lies in a $G$-orbit,  say
$\mathcal O_x$. Assume $\mathcal O_x$ is in a connected
component $M_{(H), c}$ of an orbit type $M_{(H)}$. We can
deform $\Tilde \alpha_x$ through orbits in $M_{(H), c}$ to a
loop $\Tilde \alpha'_x$ which lies in an orbit $\mathcal O'_x$,
where $\mathcal O'_x\subset M_{(H')}\subset\overline{M_{(H)}}$
and $M_{(H')}$ is in the ``direct'' closure of $M_{(H)}$. We
continue deforming the loop $\Tilde \alpha'_x$ as above inside
$\overline{M_{(H)}, c}$ until we get a loop $\Tilde \alpha^s_x$
which lies in a simply connected orbit $\mathcal
O^s_x\subset\overline{M_{(H)}, c}$. Since $\pi_1\left(\mathcal
O^s_x\right) = 1$, $[\Tilde \alpha_x] = [\Tilde \alpha'_x] =
\cdots = [\Tilde\alpha^s_x] = 1\in \pi_1(M)$. Now injectivity
follows from Lemma~\ref{liftpoint}.
\end{proof}

Here is a counter example to Proposition~\ref{propM=M/G}.

\begin{example}\label{ex2}
Consider the effective $S^1$ action on $\RP^2$ induced from the
standard $S^1$ action on $S^2$. There are $3$ orbit types: the
principal orbit type which consists of free orbits, one fixed
point, and an orbit which has $\Z_2$ as stabilizer. The $\Z_2$
orbit type does not have a simply connected orbit in its
closure.  We have
$\pi_1\left(\RP^2\right)\ncong\pi_1\left(\RP^2/S^1\right)$.
\end{example}

Next, we prove that the orbit $G\cdot x$ is simply connected if $x$
is fixed by a maximal torus.
\begin{proposition}\label{G/T}
Let $M$ be a connected smooth $G$-manifold.  If $x$ is a fixed point
of a maximal torus of $G$, then the orbit $G\cdot x$ is simply
connected.
\end{proposition}
\begin{proof}
Let $H$ be the stabilizer of $x$, then by the assumption,
$T\subseteq H$, where $T$ is a maximal torus of $G$. First, we have
$G\cdot x\approx G/H$.

Assume first that $H$ is connected.
 Since we have a surjection (see for example Theorem 7.1 in \cite{BD})
$$\pi_1(T)\twoheadrightarrow\pi_1(G),$$
we have a surjection
$$\pi_1(H)\twoheadrightarrow\pi_1(G).$$
Then by the homotopy exact sequence of the fibration
$H\overset{\iota}\hookrightarrow G\to G/H$
$$\cdots \to\pi_1(H)\to \pi_1(G)\to\pi_1\left(G/H\right)\to\pi_0(H) = 1\to\cdots,$$
 we get $\pi_1\left(G/H\right)=1$.

Now assume that $H$ is not connected. Let $H^0$ be the identity
component of $H$. By the argument above,
$\pi_1\left(G/H^0\right)=1$. The finite group $H/H^0$ acts on
$G/H^0$, and $G/H = \left(G/H^0\right)/(H/H^0)$. Let $h$ be any
representative in $H/H^0$. Then $h$ is an element of some maximal
torus of $G$. Since $T\subseteq H^0$, there exists $g\in G$ such
that $h\in g H^0 g^{-1}$. Let $h = g h^0 g^{-1}$, where $h^0\in
H^0$. Then $h\cdot (gH^0) = g h^0 g^{-1} gH^0 = g H^0$, i.e., $gH^0$
is a fixed point of $h$. Since each element in $H/H^0$ has a fixed
point in $G/H^0$, by Theorem~\ref{thmarm} (for the $H/H^0$ action on
$G/H^0$), we get $\pi_1\left(G/H\right) = 1$.
\end{proof}

\section{Hamiltonian $G$-manifolds which have a simply connected orbit}\label{section:hamG-manifolds}

 In this section, we consider Hamiltonian $G$-manifolds. We
 will show that for a Hamiltonian $G$-manifold $M$, we can
 improve the condition in Proposition~\ref{propM=M/G} (for
$G$-manifolds) to the condition that
 there exists one simply connected orbit to make the isomorphism
 $\pi_1(M)\cong\pi_1(M/G)$ to hold.

Recall Example~\ref{ex2}, in which we have an $S^1$-manifold
$M$, the orbits do not represent the same element in
$\pi_1(M)$. For a Hamiltonian $S^1$-manifold $M$, we have the
following observation, which is a key step in improving the
condition in Proposition~\ref{propM=M/G}.

\begin{lemma}\label{Spi=}
Let $M$ be a connected Hamiltonian $S^1$-manifold. Then all
$S^1$ orbits are homotopic, i.e., they represent the same
element in $\pi_1(M)$.
\end{lemma}

\begin{proof}
Let $\mathcal O$ be an $S^1$-orbit. If $\mathcal O$ is a fixed
point, since it is in the closure of the generic orbit type,
$\mathcal O$ is homotopic to any generic orbit.

Now assume $\mathcal O\subset M_{\Z_k, c}\subset M_{\Z_k}$. We
will show that $\mathcal O$ is homotopic to a generic orbit. This
is trivial if $\mathcal O$ is a generic orbit. Now assume
$\mathcal O$ is not generic, or $\Z_k$ is not the generic
stabilizer. By the local normal form theorem, a neighborhood of
$\mathcal O$ in $M$ is isomorphic to $A=S^1\times_{\Z_k}
\big(\R\times V\big)=S^1\times_{\Z_k} \big(\R\times V^H\times
W\big) \cong S^1\times_{\Z_k} \big(\R\times V^H\times
\C\times\cdots\times\C\big)$, where $\Z_k$ acts on $\R$
trivially, $V$ is a symplectic representation of $\Z_k$, $V^H$
is the subspace of $V$ fixed by $\Z_k$, and
$W\cong\C\times\cdots\times\C$ is its complementary subspace in
$V$. We have $A_{\Z_k} =S^1\times_{\Z_k} \big(\R\times
V^H\big)$. The group $\Z_k$ acts on each copy $\C$ of $W$ as a
subgroup of $S^1$, so $\C/\Z_k$ is homeomorphic to $\C$. So $A'
= S^1\times_{\Z_k} \big(\R\times V^H\times
\C/\Z_k\times\cdots\times\C/\Z_k\big)$ is homeomorphic to $A=
S^1\times_{\Z_k} \big(\R\times V^H\times
\C\times\cdots\times\C\big)$.  In both $A$ and $A'$, $\mathcal
O = S^1\times_{\Z_k} 0$. Each point in $A'= S^1\times_{\Z_k}
\big(\R\times V^H\times \C/\Z_k\times\cdots\times\C/\Z_k\big)$
is fixed by $\Z_k$. Clearly in $A'$, $\mathcal O$ is homotopic
to any $S^1$ orbit, in particular to any orbit which comes from
a generic orbit in $A$. We lift the homotopy to $A$ (through
the homeomorphism) to get a homotopy from $\mathcal O$ to a
generic orbit.
\end{proof}

A direct consequence of Lemma~\ref{Spi=} is that in a
Hamiltonian $S^1$-manifold, if the $S^1$ action has a fixed
point, then all the $S^1$-orbits are homotopically trivial.
More generally, we have the following result.

\begin{lemma}\label{Gpi=}
Let $M$ be a connected Hamiltonian $G$-manifold. Assume there
is a simply connected orbit, then each loop in each orbit is
homotopically trivial in $M$.
\end{lemma}

\begin{proof}
Let $\alpha$ be a loop contained in an orbit $\mathcal O$. If
$\mathcal O$ is simply connected, then $\alpha$ is trivial
inside $\mathcal O$.  If $\mathcal O$ is in a connected
component $M_{(H), c}$ of an orbit type $M_{(H)}$, and the
closure of $M_{(H), c}$ contains the simply connected orbit,
then we can deform $\alpha$ through orbits in $M_{(H), c}$ to a
loop $\alpha'$ in the simply connected orbit, so $[\alpha]
=[\alpha'] = 1$.

Now, assume $\mathcal O$ is in a connected component $M_{(H),
c}$ of an orbit type $M_{(H)}$, and the closure of $M_{(H), c}$
does not contain the simply connected orbit. By
Proposition~\ref{G/T}, $H\subsetneq T$, where $T$ is a maximal
torus of $G$. Moreover, $H$ is not the generic stabilizer, and
the generic stabilizer $(H') \subsetneq (H)\subsetneq (T)$.
Take $x\in\mathcal O$ such that its stabilizer is $H$.   Let
$\mathcal O_T = T\cdot x$ be the $T$ orbit. Then $\mathcal O$
fibers over $G/T$ with fiber $\mathcal O_T$, and hence
$\pi_1\left(\mathcal
O_T\right)\twoheadrightarrow\pi_1\left(\mathcal O\right)$. So
$\alpha$ is homotopic to a loop $\beta$ in $\mathcal O_T$.
Since we can split $T$ into a product of circles and the loops
in $\mathcal O_T$ are generated by the circles in
$T/H\approx\mathcal O_T$, we may assume that $\beta$ is an
$S^1\subset T/H$ orbit (or a multiple of it) and consider the
$S^1$ action. By Lemma~\ref{Spi=} for the $S^1$ action, $\beta$
is homotopic to an $S^1$-orbit $\beta'$ which is contained in
$T/H'$ and hence contained in a generic $G$-orbit. Since the
simply connected $G$-orbit is in the closure of the generic
orbit type, we can deform the loop $\beta'$ to a loop in the
simply connected $G$-orbit. So $[\alpha] =[\beta] = [\beta'] =
1$.
\end{proof}

\begin{proposition}\label{HamM=M/G}
 Let $M$ be a connected Hamiltonian $G$-manifold.
 Assume that there exists a simply connected orbit.
  Then the quotient map induces an isomorphism on $\pi_1$, i.e., $\pi_1(M)\cong\pi_1\left(M/G\right).$
 \end{proposition}

\begin{proof}
Since $G$ is connected, we have a surjection
$\pi_1(M)\twoheadrightarrow\pi_1\left(M/G\right).$

Injectivity of the map follows from Lemmas~\ref{liftpoint} and
\ref{Gpi=}.
\end{proof}

\section{Proof  of Theorem~\ref{M=M/G}}\label{section:M=M/G}
 Theorem~\ref{M=M/G}  (A)  is  Proposition~\ref{HamM=M/G}.

 Theorem~\ref{M=M/G} (B) follows from Theorem~\ref{hatG} with $\Hat G =G$,
 Theorem~\ref{M=M/G} (A) and the lemma below.

\begin{lemma}\label{levelset}
Let $M$ be a connected Hamiltonian $G$-manifold with proper
equivariant moment map $\phi$. Assume there is a simply
connected orbit $G\cdot x$.
 Then
$$\pi_1\left(\phi^{-1}(G\cdot a)\right)\cong\pi_1\left(M_a \right), \,\,\,\mbox{where}\,\,\, a=\phi(x).$$
\end{lemma}
\begin{proof}
By Theorem~\ref{retract}, there is an open invariant neighborhood
$\mathcal N$ of $\phi^{-1}(G\cdot a)$ in $M$ which equivariantly
deformation retracts to $\phi^{-1}(G\cdot a)$. By
Proposition~\ref{HamM=M/G}, $\pi_1\left(\mathcal
N\right)\cong\pi_1\left(\mathcal N/G \right)$ which is equivalent to
$\pi_1\left(\phi^{-1}(G\cdot a)\right)\cong\pi_1\left(M_a \right).$
\end{proof}

\section{Proof of Theorem~\ref{M=M/G+}}\label{section:M=M/G+}

When a maximal torus of $G$ has a fixed point,  there is a simply connected orbit by Proposition~\ref{G/T}.
 Besides Lemma~\ref{levelset}, we have Lemma~\ref{xaa}.

\begin{lemma}\label{xaa}
Let $M$ be a connected Hamiltonian $G$-manifold with proper
equivariant moment map $\phi$. If a maximal torus of $G$ has a
fixed point $x$, then
$\pi_1\left(\phi^{-1}(a)\right)\cong\pi_1(M_a)$, where
$a=\phi(x)$.
\end{lemma}

\begin{proof}
Let $\tau$ be the face of $\mathfrak t^*_+$ such that $a\in\tau$.
Let $G_a$ be the stabilizer group of $a$. Consider the $G_a$ action
on the cross section $R^{\tau}$. Then $x\in R^{\tau}$ is a fixed
point of a maximal torus of $G_a$. By Proposition~\ref{G/T}, the
orbit $G_a\cdot x$ is simply connected. Now the claim follows from
Lemma~\ref{levelset} for the $G_a$ action on $R^{\tau}$.
\end{proof}

\begin{proof}
[Proof of Theorem~\ref{M=M/G+}]
(A).  Let $T\subset G$ be the maximal torus of $G$ which  fixes $x$.
  Let $K\subset G$ be a connected subgroup. Then there is $g\in G$ such that
  $K\cap g T g^{-1}$ is a maximal torus of $K$ which fixes $g\cdot x$.   By  Proposition~\ref{G/T},
 the $K$ action has a simply connected orbit.
  Moreover, a Hamiltonian $G$-manifold is a Hamiltonian
  $K$-manifold.  Then  Theorem~\ref{M=M/G}  (A)  gives
  $$\pi_1(M)\cong\pi_1\left(M/K\right).$$
(B). By  Proposition~\ref{G/T},  the $G$ action has a simply
connected orbit $G\cdot x$. Lemmas~\ref{levelset} and
\ref{xaa}, Theorem~\ref{hatG} with $\Hat G = G$, and
Theorem~\ref{M=M/G+} (A) with $K=G$  then imply
$$\pi_1(M)\cong\pi_1\left(\phi^{-1}(G\cdot a)\right)\cong \pi_1\left(\phi^{-1}(a)\right).$$
 \end{proof}

\section{Proof of Theorem~\ref{hatG}}\label{section:hat}

 The method of removing and deforming we used in \cite{L2} to prove
  $$\pi_1(M/G)\cong\pi_1(M_a) \quad\mbox{for all $a\in \,$im$(\phi)$}$$
 for compact  manifold $M$  can be used in more  generality.  The method
 involves  local deformation retractions  as  stated  in Theorem~\ref{retract},  which relies on
 the properness of $\phi$.
The method also involves removing certain  strata from quotient
stratified spaces. If the links of these strata are connected
and simply connected,  we can use Lemma~\ref{remove:str} to do
removing. The computation of the links of the strata to be removed uses  the local normal form theorem
  and the cross section theorem. The proof of the one-connectedness of the links
   uses the property of the coadjoint action and  the normal representations of the orbit
  types from which the strata came.

 The proof of Theorem~\ref{hatG} consists of a few steps. We first prove that all the
 quotients $\phi^{-1}(G\cdot a)/\Hat G$'s  have isomorphic fundamental group, and then we prove
 that the fundamental group of $M/\Hat G$ is isomorphic to that of a
 particular quotient.\\

 \subsection{Proof of Theorem~\ref{hatG} for $G=T$}\label{section:hatT}
 \
\medskip

 In this subsection, we prove Theorem~\ref{hatG} for $G=T$, where $T$
 is a connected compact torus.  For $G=T$,
 Lemma~\ref{Tlinksame} is the key observation which
 makes the removing process work.\\

First, we have
 Lemma~\ref{Tisom} for all regular values in the same chamber.

\begin{lemma}\label{Tisom}
 Let $(M, \omega)$ be a connected Hamiltonian $T$-manifold
with  moment map $\phi$. Let $T'\subset T$ be a connected
subgroup. Then for all values $a$'s in one connected chamber of
im$(\phi)$, $\pi_1\left(\phi^{-1}(a)/T'\right)$'s are isomorphic.
\end{lemma}
\begin{proof}
This is because  for all values $a$'s  in one connected  chamber of
im$(\phi)$, $\phi^{-1}(a)$'s are $T$-equivariantly diffeomorphic.
\end{proof}

\begin{lemma}\label{Tlinksame}
  Let $(M, \omega)$ be a connected Hamiltonian $T$-manifold
with proper moment map $\phi$. Let $\Th\subset T$ be a connected
subgroup which contains the identity component of each stabilizer
group. Let $\mathcal F$ be a singular face on im$(\phi)$. Let
$\overline U$ be the closure of one open connected chamber $U$ such
that $\mathcal F \subset\overline U$.  Let $N$ be $M$ or be
$\phi^{-1}(\overline U)$. For each orbit type $M_H$ such that
$M_{H}\cap\phi^{-1}(\mathcal F)\neq\emptyset$, let $L'_H$ be the
link of $\big(M_{H}\cap\phi^{-1}(\mathcal F)\big)/\Th$  in $N/\Th$,
and let $L_H$ be the link of $\big(M_{H}\cap\phi^{-1}(\mathcal
F)\big)/T$  in $N/T$.  Then $L'_H$ is connected and simply connected
whenever $L_H$ is;  and in particular, when $H \subseteq \Th$,
$L'_H = L_H$.
\end{lemma}
\begin{proof}
Let us prove the claims for $N=M$. Let $p\in M$ be a point with
stabilizer $H$ such that $\phi(T\cdot p)\in\mathcal F$. By
Theorem~\ref{form}, a neighborhood of the (isotropic) orbit $T\cdot
p$ in $M$ is isomorphic to $A=T\times_H(\mathfrak h^{\circ}\times
V)=T\times_H(\mathfrak h^{\circ}\times W\times V^H)$, where
$\mathfrak h=\, $Lie$(H)$, $\mathfrak h^{\circ}$ is its annihilator
in $\mathfrak t^*$, $V$ is a symplectic representation of $H$, $V^H$
is the subspace of $V$ fixed by $H$ and $W$ is such that $V=W\oplus
V^H$. The moment map $\phi$ on $A$ is $\phi ([t, a, w, v])=\phi(p)+
a + \psi(w)$, where $\psi$ is the moment map of the $H$ action on
$W$. We obtain $A_H\cap\phi^{-1}(\mathcal F)=T\times_H(\R^m\times
0\times V^H)$, where $\R^m\subset\mathfrak h^{\circ}$ is a subspace
which is mapped to $\mathcal F$. Let $\R^l$ be the subspace such
that $\mathfrak h^{\circ}=\R^l\oplus\R^m$ ($H$ acts on $\mathfrak
h^{\circ}$ by the trivial coadjoint action).

First note that $\big(A_H\cap\phi^{-1}(\mathcal F)\big)/T=\R^m\times
0\times V^H$, and that $A/T=\mathfrak h^{\circ}\times V^H\times
W/H$. So $L_H=S(\R^l\times W)/H$, where $S(\R^l\times W)$ denotes a
sphere of the vector space $\R^l\times W$.

If $H\subseteq \Th$,  then $\big(A_H\cap\phi^{-1}(\mathcal
F)\big)/\Th=T/\Th\times\R^m\times 0\times V^H$, and
$A/\Th=T/\Th\times\mathfrak h^{\circ}\times V^H\times W/H$. So
$L'_H=S(\R^l\times W)/H$. So in this case, $L'_H$ is connected and
simply connected when $L_H$ is.

If $H\subseteq \Th$ except for a finite subgroup
 $\Gamma$, then
$\big(A_H\cap\phi^{-1}(\mathcal F)\big)/\Th
=T/\Th\times_H\left(\R^m\times 0\times
V^H\right)=T/\Th\times_{\Gamma}\big(\R^m\times V^H\big)$, and $A/\Th
= T/\Th\times_H \big(\mathfrak h^{\circ}\times V^H\times
W\big)=(T/\Th\times\R^m\times V^H)\times_{\Gamma}\big((\R^l\times
W)/(H\cap \Th)\big)$. So $L'_H=S(\R^l\times W)/(H\cap \Th)$. Note
that $H\cap \Th$ leaves out finitely many elements of $H$.  If $L_H$
is connected and simply connected, by the same criteria of the proof
(see our outline of proof of Lemma~\ref{Tlink0}), $L'_H$ is
connected and simply connected.

We can similarly prove the claims for $N =\phi^{-1}(\overline U)$. See
 the proof of Lemma 3.8 in
\cite{L2} for the computation of $L_H$ in this case and argue
similarly as above.
\end{proof}

For the link $L_H$  in  Lemma~\ref{Tlinksame},  in  \cite{L2}, we proved the following
facts  (see the  proof of Lemmas 7.20 and 3.8 in \cite{L2}).
 Here, for clarity, we outline the ideas of their proof.

\begin{lemma}\label{Tlink0}
Let   $L_H$ be as in Lemma~\ref{Tlinksame}.
\begin{itemize}
\item [(1)]  When $N=M$,  if
 $M_{H}$ is not principal, then $L_H$  is connected and simply connected.
\item [(2)] When $N= \phi^{-1}(\overline U)$,  $L_H$  is  always connected and simply connected.
\end{itemize}
\end{lemma}

\begin{proof} [Outline of proof] By the computations in \cite{L2} (also see the proof of
Lemma~\ref{Tlinksame} for the case when $N=M$),
$$L_H = S\big(\R^l\times W\big)/H \quad\mbox{when $N=M$};$$
$$L_H = S\big((\R^+)^l\times W\cap
\psi^{-1}(\overline{U})\big)/H \quad\mbox{when $N= \phi^{-1}(\overline U)$}.$$
Here,  $\R^l\subset \mathfrak t^*$ is a linear
subspace on which $H$ acts by the trivial coadjoint action,
$(\R^+)^l$ is the product of the corresponding nonnegative real
lines, and $\R^l = 0$ if and only if $H$ is the highest dimensional
stabilizer group of the points in $\phi^{-1}(\mathcal F)$; and $W$
is a symplectic representation of $H$ with moment map $\psi$,  $W =
0$ if and only if $H$ is the principal stabilizer.

(1) Consider $L_H$ for the case $N=M$. Assume $H$ is not the principal stabilizer, then
$W\neq 0$. Consider the following two cases. Case (a), when
$\R^l\neq 0$, $S\left(\R^l\times W\right)$ is connected and simply
connected. Since $H$ acts on $\R^l$ trivially, by
Theorem~\ref{thmarm}, $L_H$ is connected and simply connected. Case
(b), when $\R^l = 0$, then $H$ is the highest dimensional stabilizer
of the points in $\phi^{-1}(\mathcal F)$. Since $\mathcal F$ is a
singular face, $\dim(H) > 0$. If $H^0$ is the identity component of
$H$, $H^0$ has to act on $W$ nontrivially. If $\dim(W) =2$,
$S(W)/H^0$ is a single point; otherwise,
 $S(W)$ is connected and simply connected and hence so is $S(W)/H^0$.
Let $H=H^0\times\Gamma$. Then $\Gamma$ acts on $S(W)/H^0$. Since
$\Gamma$ acts on $W$ as a finite subgroup of $T$, each element in
$\Gamma$ has a fixed point in $S(W)/H^0$, so
$\left(S(W)/H^0\right)/\Gamma= S(W)/H = L_H$ is connected and simply
connected by Theorem~\ref{thmarm}.

(2) Consider $L_H$ for the case $N= \phi^{-1}(\overline U)$.  If $H$ is the principal stabilizer, then
$W=0$, and $S\big((\R^+)^l\big)$ is connected and simply connected,
and hence so is $L_H$. Now, we assume $W\neq 0$. Let $H^0$ be the
identity component of $H$. Lemmas 3.9 and 3.10 in \cite{L2} show
that  $S\big((\R^+)^l\times W\cap \psi^{-1}(\overline{U})\big)/H^0$
is connected and simply connected. Similar argument as in (1) for
the finite group $H/H^0$ action on $S\big((\R^+)^l\times
W\cap\psi^{-1}(\overline{U})\big)/H^0$ shows that $L_H$ is
connected and simply connected.
\end{proof}

Next, we aim to prove Lemma~\ref{Tisom'}. We first do the necessary
removing in Lemma~\ref{T-remove}.

\begin{lemma}\label{T-remove}
Let $(M, \omega)$ be a connected Hamiltonian $T$-manifold with
proper moment map $\phi$. Let $\Th\subset T$ be a connected subgroup
which contains the identity component of each stabilizer group.
 Let $c$ be a singular value, and let $a$ be a regular value very near $c$. Let $O$ be a small open neighborhood of $c$ containing $a$.
Let $O'$ be the intersection of $O$ with the connected open chamber containing $a$, and let
$\overline{O'}$ be its closure in $O$. Let $B$ be the set of values
in $\overline{O'} - O'$.
 Then
$$\pi_1\big(\phi^{-1}(\overline{O'})/\Th\big)\cong
\pi_1\big(\phi^{-1}(\overline{O'})/\Th-\phi^{-1}(B)/\Th\big).$$
   \end{lemma}
 \begin{proof}
 We inductively use Lemmas~\ref{Tlinksame} and \ref{Tlink0}  to remove $\phi^{-1}(B)/\Th$ from $\phi^{-1}(\overline{O'})/\Th$.
 One may refer to the proof of Lemma 3.8 in \cite{L2}.
  \end{proof}

\begin{lemma}\label{Tisom'}
Let $(M, \omega)$ be a connected Hamiltonian $T$-manifold with
proper moment map $\phi$. Let $\Th\subset T$ be a connected subgroup
which contains the identity component of each stabilizer group.  Let
$c$ be a singular value, and let $a$ be a regular value very near
$c$. Then
$\pi_1\big(\phi^{-1}(c)/\Th\big)\cong\pi_1\big(\phi^{-1}(a)/\Th\big)$.
\end{lemma}
\begin{proof}
 We take the $O$ in Lemma~\ref{T-remove} small enough so that $\phi^{-1}(O)$ and $\phi^{-1}(\overline{O'})$
 equivariantly
 deformation retracts to $\phi^{-1}(c)$ (Theorem~\ref{retract}). So
 $$\pi_1\big(\phi^{-1}(\overline{O'})/\Th\big)\cong\pi_1\big(\phi^{-1}(c)/\Th\big).$$
Since
$$\pi_1\big(\phi^{-1}(\overline{O'})/\Th-\phi^{-1}(B)/\Th\big)\cong\pi_1\big(\phi^{-1}(a)/\Th\big),$$
the claim follows from Lemma~\ref{T-remove}.
\end{proof}

 We need the following lemma when we do removing and deforming on the space $M/G$ to prove
 $\pi_1\left(M/G\right)\cong\pi_1(M_a)$ for a particular value $a$.

\begin{lemma}\label{TT_m}
Let $(M, \omega)$ be a connected Hamiltonian $T$-manifold with
proper moment map $\phi$. Let $\Th\subset T$ be a connected subgroup
which contains the identity component of each stabilizer group. Let
$\mathcal{F}$ be a singular face on im$(\phi)$. Let $O$ be a small
open neighborhood of $\mathcal{F}$ on im$(\phi)$ ($O$ does not
intersect the faces which are in the closure of $\mathcal{F}$). Let
$S$ be the set of nonprincipal orbits in $\phi^{-1}(\mathcal{F})$.
Then
$\pi_1\big(\phi^{-1}(O)/\Th\big)\cong\pi_1\big(\phi^{-1}(O)/\Th-S/\Th\big)$.
\end{lemma}
\begin{proof}
The argument is similar to the proof of Lemma 7.20 in \cite{L2} by
using Lemmas~\ref{Tlinksame} and  \ref{Tlink0} (1).
\end{proof}

\begin{proof}[Proof  of Theorem~\ref{hatG} for $G=T$.]
Using Lemmas~\ref{Tisom}, \ref{Tisom'} and \ref{TT_m}, follow the
argument used in the proof of Theorem 1.6 in \cite{L2} for the case
$G=T$.
\end{proof}

 \subsection{Proof of  Theorem~\ref{hatG} for nonabelian
 $G$}\label{section:hatG}
 \
 \medskip

In this subsection, we prove Theorem~\ref{hatG} for nonabelian group
actions. In this case, the proof is more technical than the case
when an abelian group acts. We need to work in the symplectic cross
sections and then work in $M$. Similar to Lemma~\ref{Tlinksame},
Lemma~\ref{Glinksame} is the key observation which makes the
removing process work.

\begin{lemma}\label{Glinksame}
Let $(M, \omega)$ be a connected Hamiltonian $G$-manifold with
proper moment map $\phi$, where $G$ is nonabelian.   Let $\Gh\subset
G$ be a connected subgroup which contains the identity component of
each stabilizer group. Let $C$ be the central face of $\mathfrak
t^*_+$, and assume that $C\cap\mbox{im}(\phi)\neq \emptyset$ and
that $C$ is not the only face of $\mathfrak t^*_+$ which intersects
$\mbox{im}(\phi)$.
 For each orbit type $M_{(H)}$ such that $M_{(H)}\cap\phi^{-1}(C)\neq\emptyset$,
 let $L'_H$ be the
link of  $\left(M_{(H)}\cap\phi^{-1}(C)\right)/\Gh$ in $M/\Gh$,  and  $L_H$ be the
link of $\left(M_{(H)}\cap\phi^{-1}(C)\right)/G$ in $M/G$.
 Then $L'_H$ is connected and simply connected whenever $L_H$ is; and in
particular, $L'_H = L_H$ when $H\subseteq \Gh$.
\end{lemma}
 \begin{proof}
By the local normal form theorem (Theorem~\ref{form}), a
neighborhood in $M$ of an orbit $\mathrm O$ in $\phi^{-1}(C)$ with
stabilizer group
  $(H)$ is isomorphic to
$A=G\times_H (\mathfrak{h}^{\circ}\times V)=G\times_H
(\mathfrak{h}^{\circ}\times W\times V^H)$, where
$\mathfrak{h}^{\circ}$ is the annihilator of
 $\mathfrak{h}=\,$Lie$(H)$ in $\mathfrak{g}^*=\,$Lie$^*(G)$, $V$ is a symplectic representation
 of $H$,
$V^H$ is the subspace of $V$ fixed by $H$ and $W$ is such that
$V=W\oplus V^H$. The moment map $\phi$ on $A$ is $\phi \left([g, a,
w, v]\right)=Ad^*(g)\left(\phi(\mathrm O)+ a + \psi(w)\right)$,
where $\psi$ is the moment map of the $H$ action on $W$. So
$A_{(H)}\cap\phi^{-1}(C) = G\times_H(\R^m\times 0\times V^H)$, where
$\R^m\subset\mathfrak h^{\circ}$ is the subspace which is mapped to
$C$ and therefore on which $H$ acts trivially.  Let
$\R^l\subset\mathfrak h^{\circ}$ be the subspace such that
$\mathfrak h^{\circ}=\R^l\oplus\R^m$.

Then $\left(A_{(H)}\cap\phi^{-1}(C)\right)/G=\R^m\times 0\times V^H$, and
$A/G=\R^m\times 0\times V^H\times (\R^l\times W)/H$. So
$L_H=S(\R^l\times W)/H$.

If $H\subseteq \Gh$, then
$\left(A_{(H)}\cap\phi^{-1}(C)\right)/\Gh=G/\Gh\times\R^m\times V^H$, and
$A/\Gh = G/\Gh\times\R^m\times V^H \times (\R^l\times W)/H$.
So $L'_H=S(\R^l\times W)/H=L_H$.

If the identity component of $H$ is contained in $\Gh$, but
$H\nsubseteq \Gh$, then $\left(A_{(H)}\cap\phi^{-1}(C)\right)/\Gh$
$=G/\Gh\times_H (\R^m\times V^H)=(\R^m\times V^H)\times
G/\Gh\times_H 0$, and $A/\Gh=G/\Gh\times_H (\R^m\times V^H\times
\R^l\times W)=(\R^m\times V^H)\times G/\Gh\times_H (\R^l\times W).$
So $L'_H=S(\R^l\times W)/(H\cap \Gh)$.

  The claim for the case
$H\subseteq \Gh$ follows immediately. For the case $H\nsubseteq
\Gh$, note that $H\cap \Gh$ leaves out finitely many elements of
$H$.  If  $L_H$ is connected and simply
connected,  by the criteria of the proof,   $L'_H$ is connected and simply connected. We refer to our
outline of proof of Lemma~\ref{Glink0}  below and the proof of Lemma 6.19
in \cite{L2}  for details.
\end{proof}

In \cite{L2}, we proved the one-connectedness of the link $L_H$ occuring in Lemma~\ref{Glinksame}.
 Here, in the outline of proof, we give the proof  for one case. We refer to the proof of Lemma 6.19
in \cite{L2} for details.

\begin{lemma}\label{Glink0}
 The link $L_H$  in Lemma~\ref{Glinksame}  is connected and
simply connected.
\end{lemma}

\begin{proof} [Outline of proof]
We saw in the proof of Lemma~\ref{Glinksame} that
$$L_H = S(\R^l\times W)/H.$$
 If we split $G=K\times T_c$, where $K$ is
semisimple and $T_c$ is abelian, then $\R^l\subset\mbox{Lie}^*(K)$
is a subspace on which $H$ acts by the coadjoint action; and,
$\dim(\R^l)\geq 2$ unless $H=K\times T_1$, where $T_1\subset T_c$ is
a subgroup, and in the latter case, $\R^l = 0$.

Let us only consider the case when $H=K\times T_1$. Then $L_H =
S(W)/(K\times T_1)$ in this case. Since $\mbox{im}(\phi)$ intersects
other faces other than $C$, $W$ is a nontrivial $K$ representation
with real dimension at least $4$. Hence $S(W)$ is connected and
simply connected. Let $H^0 = K\times T_1^0$ be the identity
component of $H$. Then $S(W)/H^0$ is connected and simply connected.
The finite subgroup $T_1/T_1^0$ acts on $S(W)/H^0$. Since
$T_1/T_1^0$ acts on $S(W)$ as a subgroup of a torus,  each element
must fix a point in $S(W)/H^0$. Hence, by Theorem~\ref{thmarm},
$L_H=\left(S(W)/H^0\right)/(T_1/T_1^0)$ is connected and simply
connected.
\end{proof}

Using Lemmas~\ref{Glinksame} and \ref{Glink0},  we can remove $\phi^{-1}(C)/\Gh$ from a
quotient as in Lemma~\ref{Glink}, where $C$ is the central face of
$\mathfrak t^*_+$.

\begin{lemma}\label{Glink}
Let $(M, \omega)$ be a connected Hamiltonian $G$-manifold with
proper moment map $\phi$, where $G$ is nonabelian.   Let $\Gh\subset
G$ be a connected subgroup which contains the identity component of
each stabilizer group. Let $C$ be the central face of $\mathfrak
t^*_+$, and assume that $C\cap\mbox{im}(\phi)\neq \emptyset$ and
that $C$ is not the only face of $\mathfrak t^*_+$ which intersects
$\mbox{im}(\phi)$. Then
$$\pi_1\big(M/\Gh\big)\cong \pi_1\big(M/\Gh-\phi^{-1}(C)/\Gh\big).$$
Similarly, if $O$ is a small open invariant  neighborhood of $C$ in
$\mathfrak g^*$, then
$$\pi_1\big(\phi^{-1}(O)/\Gh\big)\cong \pi_1\big(\phi^{-1}(O)/\Gh-\phi^{-1}(C)/\Gh\big).$$
\end{lemma}
\begin{proof}
By Lemmas~\ref{Glinksame}  and \ref{Glink0},  for each possible $H$, the link of
$\left(M_{(H)}\cap\phi^{-1}(C)\right)/\Gh$ in $M/\Gh$ is connected
and simply connected. Therefore, using Lemma~\ref{remove:str}, we
can inductively remove the strata of $\phi^{-1}(C)/\Gh$ from $M/\Gh$
or from $\phi^{-1}(O)/\Gh$.
\end{proof}

For any other removing, we need to work in a suitable cross section,
and use Lemmas~\ref{Glinksame}  and \ref{Glink0} for a subgroup action. Then, using the
equivariance of $\phi$, we do corresponding removing in the quotient
of $M$ or in the quotient of an invariant subset of $M$. We will
need Lemma~\ref{link}.

\begin{lemma}\label{link}
Let $(M, \omega)$ be a connected Hamiltonian $G$-manifold with
proper moment map $\phi$, where $G$ is nonabelian.   Let $\Gh\subset
G$ be a connected subgroup which contains the identity component of
each stabilizer group.
\begin{enumerate}
\item  Let $\tau$ be a non-principal
  face of  $\mathfrak t^*_+$  such that
  $\tau\cap\mbox{im}(\phi)\neq\emptyset$ and let $R^{\tau}$ be the corresponding cross section. Then  for each
 $M_{(H)}\cap\phi^{-1}(\tau)$, the link of
  $\left(M_{(H)}\cap\phi^{-1}(\tau)\right)/(G_{\tau}\cap \Gh)$ in $R^{\tau}/(G_{\tau}\cap \Gh)$
  is the same as the link of  $\left(M_{(H)}\cap\phi^{-1}(G\cdot\tau)\right)/\Gh$ in $\big(G\cdot R^{\tau}\big)/\Gh$.
\item  Let $\tau^P$ be the principal face and let $R^{P}$
be the principal cross section. Let
 $c\in\tau^P$ be a singular value. Then  for each  $M_{(H)}\cap\phi^{-1}(c)$, the link of
$\left(M_{(H)}\cap\phi^{-1}(c)\right)/(T\cap \Gh)$ in
 $R^{P}/(T\cap \Gh)$ is the same
 as the link of  $\left(M_{(H)}\cap\phi^{-1}(G\cdot c)\right)/\Gh$ in $\big(G\cdot R^{P}\big)/\Gh$.
\end{enumerate}
\end{lemma}
\begin{proof}
 First note that, if a point with stabilizer group $H$ is mapped to a face $\tau$, then since $\phi (g\cdot m)=Ad^*(g)\cdot\phi(m)$,
$H\subset G_{\tau}$.

(1) The manifold $G\cdot R^{\tau}$ is a fibration over the coadjoint orbit $G/G_{\tau}$ with fiber $R^{\tau}$; and
correspondingly, each  $M_{(H)}\cap (G\cdot R^{\tau})$ is a fibration over $G/G_{\tau}$ with fiber
$M_{(H)}\cap R^{\tau}$, moreover,  each $M_{(H)}\cap\phi^{-1}(G\cdot \tau)$ is a fibration over $G/G_{\tau}$ with fiber
$M_{(H)}\cap\phi^{-1}(\tau)$.  The claim follows from the equivariance of $\phi$.

The proof of (2) is similar.
\end{proof}

\begin{remark}
Note that if $\tau\subset\mathfrak t^*_+$ is a face and $G_{\tau}$
is its stabilizer, the subgroup $G_{\tau}\cap \Gh$ of $G_{\tau}$ is
connected  since $G_{\tau}$ and $\Gh$ are both connected.
\end{remark}

Let us first consider the values of $\phi$ on the principal face. We
obtain Lemma~\ref{Gisom}.

\begin{lemma}\label{Gisom}
Let $(M, \omega)$ be a connected Hamiltonian $G$-manifold with
proper moment map $\phi$, where $G$ is nonabelian.   Let $\Gh\subset
G$ be a connected subgroup which contains the identity component of
each stabilizer group.   Let  $\tau^P\subset\mathfrak t^*_+$
   be the principal face. Then
 $\pi_1\big(\phi^{-1}(G\cdot\tau^P)/\Gh\big)\cong\pi_1\big(\phi^{-1}(G\cdot a)/\Gh\big)$ for
  all $a\in\tau^P$.
  \end{lemma}
\begin{proof}
First, consider the space $\phi^{-1}(\tau^P)$ with the $T$ action,
where $T$ is a maximal torus of $G$. Similar to the proof of the
theorem for abelian group actions, by using removing and deforming
in the space $\phi^{-1}(\tau^P)/(\Gh\cap T)$, we get
$$\pi_1\big(\phi^{-1}(\tau^P)/(\Gh\cap T)\big)\cong\pi_1\big(\phi^{-1}(a)/(\Gh\cap
T)\big)$$ for some {\em particular} $a\in \tau^P$. Then by
Lemma~\ref{Tlinksame}  and Lemma~\ref{link} (2),  we can do
corresponding removing and deforming in the space
$\phi^{-1}(G\cdot\tau^P)/\Gh$, and we arrive at
$$\pi_1\big(\phi^{-1}(G\cdot\tau^P)/\Gh\big)\cong\pi_1\big(\phi^{-1}(G\cdot a)/\Gh\big)$$ for
this {\em particular} $a\in\tau^P$.

For any two  values $a$ and $b$ in the same connected chamber on
$\tau^P$,
 $$\pi_1\big(\phi^{-1}(G\cdot a)/\Gh\big)\cong \pi_1\big(\phi^{-1}(G\cdot b)/\Gh\big)$$
since  $\phi^{-1}(G\cdot a)$ and  $\phi^{-1}(G\cdot b)$ are
equivariantly diffeomorphic. Now, let $c\in\tau^P$ be a value on a
singular face. Take a value $a\in\tau^P$ in a connected open face
very close to $c$. Take $\overline {O'}\subset \tau^P$ as in
Lemma~\ref{T-remove}. Using Lemma~\ref{link} (2) again, we can do
removing in $\phi^{-1}(G\cdot\overline{O'})/\Gh$ corresponding to
the removing in $\phi^{-1}(\overline{O'})/(T\cap \Gh)$; and we
arrive at
$$\pi_1\big(\phi^{-1}(G\cdot a)/\Gh\big)\cong\pi_1\big(\phi^{-1}(G\cdot c)/\Gh\big).$$
\end{proof}

Next, we consider a value $c$ on a nonprincipal face and a nearby
generic value $a$ on the principal face and prove
Lemma~\ref{Gisom'}. We first do the necessary removing in
Lemma~\ref{remove}.

\begin{lemma}\label{remove}
 Let $(M, \omega)$ be a connected Hamiltonian $G$-manifold with
proper moment map $\phi$, where $G$ is nonabelian.   Let $\Gh\subset
G$ be a connected subgroup which contains the identity component of
each stabilizer group. Let $c\in\tau$ be a value, where $\tau$ is a
face of
   $\mathfrak t^*_+$ such that  $\tau\neq\tau^P$, and let $a$ be a generic value on
    $\tau^P$ very near $c$.
Let $O\subset\mathfrak{g}^*$ be a small open invariant neighborhood
of $c$ containing $a$.
    Let $B$ be the set of values in $O\cap\mathfrak t^*_+$  other than those on the open connected
    chamber of generic values on  $\tau^P$ containing $a$.
   Then
   $$\pi_1\big(\phi^{-1}(O)/\Gh\big)\cong\pi_1\big(\phi^{-1}(O)/\Gh-\phi^{-1}(G\cdot
   B)/\Gh\big).$$
   \end{lemma}
 \begin{proof}
Consider the cross section $R^{\tau}$ where $G_{\tau}$ acts. Note that $\tau$ lies in the central dual Lie algebra
of $G_{\tau}$.  Using Lemmas~\ref{Glink} and \ref{link}, we have
$$\pi_1\big(\phi^{-1}(O)/\Gh\big)\cong\pi_1\big(\phi^{-1}(O)/\Gh-\phi^{-1}(G\cdot \tau)/\Gh\big).$$
For other non-principal faces $\tau'$'s,  we use the cross section
theorem and Lemma~\ref{link} to inductively remove
$\phi^{-1}(G\cdot\tau')/\Gh$'s. If there are singular faces on
$O\cap\tau^P$,  then we use Lemmas~\ref{Tlinksame} and \ref{link} to
remove the rest (deforming may also be needed). If further detail is
prefered, one may refer to the proof of Lemma 6.18 in \cite{L2}.
\end{proof}

   \begin{lemma}\label{Gisom'}
Let $(M, \omega)$ be a connected Hamiltonian $G$-manifold with
proper moment map $\phi$, where $G$ is nonabelian.   Let $\Gh\subset
G$ be a connected subgroup which contains the identity component of
each stabilizer group. Let $c\in\tau$ be a value, where $\tau$ is a
face of
   $\mathfrak t^*_+$ such that  $\tau\neq\tau^P$, and let $a$ be a generic value on
    $\tau^P$ very near $c$. Then
 $\pi_1\big(\phi^{-1}(G\cdot c)/\Gh\big)\cong\pi_1\big(\phi^{-1}(G\cdot a)/\Gh\big)$.
   \end{lemma}
\begin{proof}
In Lemma~\ref{remove}, we take $O$ small enough such that
$\phi^{-1}(O)$ equivariantly deformation retracts to
$\phi^{-1}(G\cdot c)$. Then
$$\pi_1\big(\phi^{-1}(O)/\Gh\big)\cong\pi_1\big(\phi^{-1}(G\cdot
c)/\Gh\big).$$ Since
$$\pi_1\big(\phi^{-1}(O)/\Gh-\phi^{-1}(G\cdot
B)/\Gh\big)\cong\pi_1\big(\phi^{-1}(G\cdot a)/\Gh\big),$$ the claim
follows from Lemma~\ref{remove}.
\end{proof}

\begin{proof} [Proof of Theorem~\ref{hatG} for nonabelian $G$] We use
Lemma~\ref{Glink}, the cross section theorem and Lemma~\ref{link} to
inductively remove $\phi^{-1}(G\cdot \tau)/\Gh$ from $M/\Gh$
for the faces $\tau$'s other than the principal face $\tau^P$ which contain values of $\phi$.
 Assume we have now
$\pi_1(M/\Gh)\cong\pi_1\big(\phi^{-1}(G\cdot\tau^P)/\Gh \big)$. Then
by Lemma~\ref{Gisom}, $\pi_1\big(\phi^{-1}(G\cdot\tau^P)/\Gh
\big)\cong\pi_1\big( \phi^{-1}(G\cdot a)/\Gh\big)$ for $a\in\tau^P$.
The claim that the $\pi_1\big( \phi^{-1}(G\cdot a)/\Gh\big)$'s are
isomorphic for all $a\in\mbox{im}(\phi)$ follows from
Lemma~\ref{Gisom'}.
\end{proof}


\begin{thebibliography}{99}
\bibitem{A}  M. A. Armstrong, {\em Calculating the fundamental group of
  an orbit space}, Proc. Amer. Math. Soc., {\bf 84} (1982), no. 2, 267-271.
\bibitem{AT} M. Atiyah, {\em Convexity and commuting Hamiltonians}, Bull. Lond. Math. Soc. {\bf 14} (1982),
1-15.
\bibitem {BD} T. Br\"ocker and T. Tom Dieck, {\em Representations of
compact Lie groups}, Graduate Texts in Mathematics, Vol. {\bf 98},
Springer, 1985.
\bibitem {Br} G. Bredon, {\em Introduction to Compact Transformation Groups}, Academic Press, New York, 1972.
\bibitem {GLS} V. Guillemin, E. Lerman, and S. Sternberg, {\em Symplectic fibrations and multiplicity diagrams},
 Cambridge University Press, Cambridge, 1996.
\bibitem{GS0} V. Guillemin and S. Sternberg, {\em Convexity properties of the moment mapping},
 Invent. Math. {\bf 67} (1982), no. 3, 491--513.
\bibitem {GS} V. Guillemin and S. Sternberg, {\em Symplectic techniques in physics},
Cambridge University Press, Cambridge, 1990.
\bibitem {GS1} V. Guillemin and S. Sternberg, {\em A normal form for the moment map},
Differential geometric methods in mathematical physics, (S. Sternberg, Ed.) Reidel, Dordrecht, Holland, 1984.
\bibitem{GSD} L. Godinho and M. E. Sousa-Dias, {\em The fundamental group of $S^1$-manifolds},
  Canad. J. Math. {\bf 62} (2010) 1082-1098.
\bibitem {K} F. C. Kirwan, {\em Convexity properties of the moment mapping III}, Invent. Math. {\bf 77} (1984), 547-552.
\bibitem {Ler} E. Lerman, {\em Gradient flow of the norm squared of a moment map}, Enseign. Math., {\bf 51} (2005), no. 1-2, 117-127.
\bibitem {LMTW} E. Lerman, E. Meinrenken, S. Tolman, and C. Woodward, {\em Nonabelian convexity by symplectic cuts}, Topology {\bf 37} (1998), no. 2, 245--259.
 \bibitem {L1} H. Li, {\em $\pi_1$ of Hamiltonian $S^1$-manifolds}, Proc.  Amer. Math. Soc. {\bf 131} (2003), no. 11, 3579-3582.
 \bibitem {L2} H. Li, {\em The fundamental group of symplectic manifolds with Hamiltonian Lie group actions},  J. of Symplectic  Geom., Vol. {\bf 4} (2007), No. 3, 345-372.
\bibitem {M} C. M. Marle, {\em Le voisinage d'une orbite d'une action hamiltonienne d'un group de Lie}, In: S\'eminaire sud-rhodanien
de g\'eom\'etrie, $\Pi$ (Lyon, 1983), 19-35, Travaux en cours, Hermann, Paris, 1984.
\bibitem {SL} R. Sjamaar and E. Lerman, {\em Stratified symplectic spaces and reduction}, Ann. of Math.
  {\bf 134} (1991), no. 2, 375--422.
\bibitem {W} C. Woodward, {\em  The Yang-Mills heat flow on the moduli space of framed bundles on a surface},
Amer. J. Math. {\bf 128} (2006), no. 2, 311-359.
 \end{thebibliography}
\end{document}